\documentclass{birkmult}
\usepackage[dvips]{epsfig}
\usepackage{amsmath}
\usepackage{amsfonts}
\usepackage[psamsfonts]{amssymb}
\usepackage{eufrak}
\usepackage[mathscr]{eucal}
\usepackage[T1,T2A]{fontenc}
\usepackage[cp1251]{inputenc}
\usepackage[russian,english]{babel}
\newtheorem{theorem}{Theorem}%[section]
\newtheorem{lemma}{Lemma}%[section]
\newtheorem{definition}{Definition}%[section]
\newtheorem{remark}{Remark}%[section]
\newtheorem*{nonumtheorem}{Theorem}
\def\theequation{\arabic{section}.\arabic{equation}}

\title[Stieltjes Functions and Hurwitz Stable Entire Functions]{Stieltjes Functions and Hurwitz Stable \\ Entire Functions.}
\author{Victor Katsnelson}
\address{Department of Mathematics\\
 the Weizmann Institute\\
  Rehovot 76100\\
   Israel}
\email{victor.katsnelson@weizmann.ac.il,
victorkatsnelson@gmail.com}%

\dedicatory{Dedicated to Iosif Vladimirovich Ostrovskii.}

\begin{document}
\begin{abstract}
\label{LPS}
The concept of stability, originally introduced for polynomials,
will be extended to apply to the class of entire functions.
This generalization
will be called Hurwitz stablility and
the class of Hurwitz stable functions will serve
as the main focus of this paper.
A first theorem will show how, given a function
of either of the Stieltjes classes, a Hurwitz stable function
might be constructed.
A second approach to constructing Hurwitz stable functions,
based on using additional functions from the Laguerre-P\'{o}lya class,
will be presented in a second theorem.
\end{abstract}
\maketitle

\section{Introduction}
\label{Sec1}
\setcounter{equation}{0}

The main focus of this paper will be on a particular
generalization of the idea of polynomial stability
(here, we mean polynomials with complex coefficients).

%
% In this paper we largely focus on establishing and exploring
% the relationship between stable polynomials and holomorphic Stieltjes functions.

\emph{Stable polynomials} are polynomials which only have roots
in the open left half-plane \(\{z:\,\textup{Re}\,z<0\}\).\
These polynomials are important in automatic control theory.
The well-known Routh-Hurwitz Criterium allows for a complete
characterization of stable polynomials in
terms of their coefficients.

There was already
much interest in generalizing
the Routh-Hurwitz Stability Criterion for
suitable classes of entire functions
in the early part of the last century.
(See, for instance, the monographs \cite{CM} and \cite{L1}.)
We now introduce the stability concept for entire functions,
which we will be using:

\begin{definition}
\label{DeHuS} An entire function \(\Phi(z)\) is called
\emph{Hurwitz stable} if \(\Phi(z)\) grows  not more than
exponentially, i.e.
\begin{equation*}
\varlimsup_{z\to\infty}|z|^{-1}\ln|\Phi(z)|<\infty,
\end{equation*}
and satisfies the conditions:
\begin{enumerate}
\item[\textup{1.}] All roots of the function \(\Phi(z)\) lie in the
open left half-plane \(\{z:\,\textup{Re}\,z<0\}\);
\item[\textup{2.}]
\(h_{\Phi}(0)\geq{}h_{\Phi}(\pi)\), where \(h_{\Phi}(\theta)\) is
the indicator function of \(\Phi\), i.e.
\begin{equation*}
h_{\Phi}(\theta)=\varlimsup_{r\to\infty}r^{-1}\ln|\Phi(re^{i\theta})|.
\end{equation*}
\end{enumerate}
\end{definition}

A variant of stability for entire functions, where
the left half-plane is replaced by the upper half-plane, can be found
in publications on Hilbert spaces of entire functions by L. de Branges
(see, e.g., his monograph \cite{deB}).

The central goal of this paper is to present two methods for
constructing particular classes of Hurwitz stable entire functions.
At the outset, we begin with specific classes of functions holomorphic
in \(\mathbb{C} \setminus ( - \infty, \, 0 ]\), namely, the functions belonging
to the Stieltjes classes
\(\mathbf{S}\) and \(\mathbf{S^{\boldsymbol{-1}}}\), introduced in
Definition \ref{StCl}.
The importance of these function classes first became apparent
through Stieltjes' classical work \cite{St} on what is now known as the
Stieltjes Moment Problem.  Stieltjes' method involved
associating each non-negative finite measure
\(d\sigma(\lambda)\) supported on \([0,\infty)\)
with a function
\begin{math}
\psi_\sigma(z)=\int\limits_{0}^{+\infty}\frac{d\sigma(\lambda)}{\lambda+z},
 \ z\in\mathbb{C}\setminus(-\infty,0],
\end{math}
which is holomorphic in \(\mathbb{C} \setminus ( - \infty, \, 0 ]\).
This function \(\psi_\sigma(z)\) is called the \emph{Cauchy transform}
of the measure \(d\sigma\).
%
% The notion of a stable polynomial can be
% extended so that it applies to a certain class of entire functions.
% (See Definition \ref{DeHuS}, below.)
%
% \emph{Stieltjes functions}
% (or functions belonging to the \emph{Stieltjes class}) are related to the
% Stieltjes Moment Problem.  Given a non-negative finite measure
% \(d\sigma(\lambda)\) supported on \([0,\infty)\), we associate
% the function
% \begin{math}
% \psi(z)=\int\limits_{0}^{+\infty}\frac{d\sigma(\lambda)}{\lambda+z},
%  \ z\in\mathbb{C}\setminus(-\infty,0],
% \end{math}
% with this measure. (\(\psi\) is called the
% \emph{Cauchy transform} of the measure \(d\sigma\).)
The Cauchy transforms of finite non-negative measures supported on
\([0,\infty)\) are functions \(\psi\) holomorphic in the domain
\(\mathbb{C}\setminus(-\infty,0]\) and possessing
certain positivity properties there. These positivity properties are:
\[\psi(x)\geq0 \ \textup{for} \ x>0,\ \textup{Im}\,\psi(z)\leq0\
\textup{for}\ \textup{Im}\,z>0,\ \ \textup{Im}\,\psi(z)\geq0\
\textup{for}\ \textup{Im}\,z<0.\]
% The \emph{Stieltjes class} is
% the class of functions \(\psi(z)\) which are holomorphic in the
% domain \(\mathbb{C}\setminus(-\infty,0]\) and that satisfy the above
% positivity conditions. (See Definition \ref{StCl},
% below.)
It can be shown that the Stieltjes class coincides with the class of Cauchy
transforms of finite non-negative measures supported on \([0,\, +\infty)\).

The data for the Stieltjes Moment Problem is given as a sequence
\(\{s_k\}_{0\leq{}k<\infty}\) of real numbers. A solution of the
Stieltjes Moment Problem is any finite \emph{non-negative} measure
\(d\sigma(\lambda)\) supported on the positive half-axis
\([0,+\infty)\) such that the moments of this measure coincide with their respective
terms of the data sequence, i.e.:
\(\int\limits_{0}^{+\infty}\lambda^kd\sigma(\lambda)=s_k,\,0\leq{}k<\infty\).
The Stieltjes Moment Problem (which, in its original form,
is an integral representation problem) can be reformulated as a classical
interpolation problem in the Stieltjes class of functions. The
measure \(d\sigma(\lambda)\) is a solution of the above-described
Stieltjes moment problem if, and only if, its Cauchy transform
\(\psi\) (which is a function from the Stieltjes class) admits the
asymptotic expansion
\begin{math}
\psi(z)=\frac{s_0}{z}-\frac{s_1}{z^2}+\,\cdots\,+(-1)^{n}\frac{s_n}{z^{n+1}}+\,\cdots
\end{math}
near the point \(z=-\infty\).

Functions from the Stieltjes class are also related to positive
operators in Hilbert spaces. Let \(\mathfrak{H}\) be a Hilbert
space with the scalar product \(\langle\,.\,,.\,\rangle\). Furthermore,
let \(A\) be a positive self-adjoint operator in \(\mathfrak{H}\)
and \(e\in\mathfrak{H}\) with \(e\not=0\). The function
\(\psi(z)=\langle(A+zI)^{-1}e,e\rangle\) then belongs to the Stieltjes
class. This function \(\psi\) is, moreover, the Cauchy transform of the
measure \(d\sigma(\lambda)=\langle dE(\lambda)e,e\rangle\), where
\(dE(\lambda)\) is the resolution of identity of the operator
\(A\).

The simplest of our results can be formulated as follows:\\
\emph{Let \(P(z)=\sum\limits_{k}p_kz^k\) be a polynomial with only
negative roots, \(\psi(z)\) be a function from the Stieltjes class
and \(P_{\psi}(z)\) be the polynomial
\(P_{\psi}(z)=\sum\limits_{k}p_k\psi(k+1)z^k\). Then the
polynomial \(P_{\psi}\) is stable.}

This paper is organized as follows.
In Section \ref{Sec2}, we present a summary of
basic facts on functions belonging to the Stieltjes
classes. This material plays an important role in the
proofs of our main results.

Starting with a function from one of the Stieltjes classes,
\(\mathbf{S}\) or \(\mathbf{S^{\boldsymbol{-1}}}\),
we develop an approach to constructing Hurwitz stable
entire functions in Section \ref{Sec3} (see Theorem \ref{Te1}).

Section \ref{Sec4} deals with a generalization of the methods
(for constructing Hurwitz stable entire functions) of Section \ref{Sec3}.
This generalized approach (see Theorem \ref{Te2}) is based on
the additional use of a Laguerre-P\'{o}lya class entire function.
If this Laguerre-P\'{o}lya function is, in particular, chosen as
the exponential function, we again arrive at Theorem \ref{Te1}.
Our proof of Theorem \ref{Te2}, however, relies on
Theorem \ref{Te1}.

Our work here was motivated by the work of T.\,Craven and G.\,Csordas in \cite{CC}.
It was through \cite{CC} that we became interested in meromorphic Laguerre multiplier
sequences. Since meromorphic functions of the form \eqref{MerSF}
belong to the class \(\mathbf{S}\), our Theorems \ref{Te1} and \ref{Te2}
can be used to construct a class of meromorphic Laguerre multiplier
sequences. We hope soon to present a construction of this kind in a upcoming paper.
%
%
% We now extend the concept of stability (which, up until
% now, we have only used for polynomials) to apply to
% the class of entire functions. This generalization
% will be called \emph{Hurwitz stablility} and
% will, naturally, lead us to focus on the class of Hurwitz
% stable functions.
% In Theorem \ref{Te1}, we will first see how, given a function
% of either of the Stieltjes classes, a Hurwitz stable function
% might be constructed.
% A second method for constructing Hurwitz stable functions,
% based on using functions of the Laguerre-P\'{o}lya class,
% will be presented in Theorem \ref{Te2}.

%%%%%%%%%%%%%%%%%%%%%%%%%%%%%%%%%%%%%%%%%%%%%%%%%%%%%%%%%%%%%%%%%%%%%%%%%%%%
\section{Some Basic Facts on the Stieltjes Classes}
\label{Sec2}
\setcounter{equation}{0}
The Stieltjes classes
\(\mathbf{S}\) and \(\mathbf{S^{\boldsymbol{-1}}}\) are defined as follows:
\begin{definition}
\label{StCl}
Let \(\psi\) be a function holomorphic in the domain \(\mathbb{C}\setminus(-\infty,\,0]\).
\begin{enumerate}
\item[\textup{1)}.]
\emph{The function \(\psi\) belongs to the class \(\mathbf{S}\)} if \(\psi\) satisfies the
conditions
\begin{equation*}
\psi(x)>0 \textup{ for } 0<x<\infty \ \textup{ and } \
\textup{Im}\,\psi(z)\leq0 \textup{ for } \textup{Im}\,z\geq0\,;
\end{equation*}
\item[\textup{2)}.]
\emph{The function \(\psi\) belongs to the class \(\mathbf{S^{\boldsymbol{-1}}}\)} if \(\psi\) satisfies the
con\-di\-ti\-ons
\begin{equation*}
\psi(x)>0 \textup{ for } 0<x<\infty \ \textup{ and } \ \textup{Im}\,\psi(z)\geq0
\textup{ for } \textup{Im}\,z\geq0\,.
\end{equation*}
\end{enumerate}
\end{definition}

\begin{remark}
\label{nonco}
The only functions \(\psi\) belonging to
both classes \(\mathbf{S}\) and \(\mathbf{S^{\boldsymbol{-1}}}\)
are the identically constant functions. From here on, we will
thus assume that \(\psi(z)\not\equiv\textup{const}\).
\end{remark}

\begin{remark}%
A function \(\psi\in\mathbf{S}\) decreases on the
positive half-axis and takes non-negative values there.
Therefore, the limit \(\lim_{x\to+0}\psi(x)\) (which may be infinite) exists.
We define the value \(\psi(0)\) of \(\psi\) at the point
\(z=0\) to be this limit, even if \(\psi\) is not holomorphic at \(z=0\).

A function \(\psi\in\mathbf{S}^{-1}\) increases on the
positive half-axis and takes non-negative values there. Therefore,
the (finite) limit \(\lim_{x\to+0}\psi(x)\) exists. We define
the value \(\psi(0)\) of \(\psi\) at the point
\(z=0\) to be this limit, even if \(\psi\) is not holomorphic at \(z=0\):
\begin{equation}
\label{psiaz}
\psi(0)\stackrel{\textup{\tiny{def}}}{=}\lim_{t\to+0}\psi(t)\,.
\end{equation}
Clearly,
\begin{subequations}
\label{vaz}
\begin{align}
\label{vazp}
 0<\psi(0)\leq+\infty, \ \ &\textup{if } \psi\in\mathbf{S},\\
\label{vazm}
 0\leq\psi(0)<+\infty, \ \ &\textup{if } \psi\in\mathbf{S}^{-1}.
\end{align}
\end{subequations}
\end{remark}

\vspace{3.0ex}
\noindent
Functions \(\psi\) of the form
\begin{subequations}
\label{Spe}
\begin{equation}%
\label{SpeS}
\psi(z)=a+\frac{b}{z}\,,\ \textup{where}\  a,\,b \ \  \textup{are  constants}, \ a\geq0,\,b>0,
\end{equation}
belong to the class \(\mathbf{S}\).
Functions \(\psi\) of the form
\begin{equation}%
\label{SpeSm}
\psi(z)=a+bz,\ \textup{where} \ a,\,b \ \  \textup{are  constants}, \ a\geq0,\,b>0,
\end{equation}
\end{subequations}
belong to the class \(\mathbf{S}^{\boldsymbol{-1}}\).
Functions \(\psi\in\mathbf{S}\) of the form \eqref{SpeS} and
functions \(\psi\in\mathbf{S}^{\boldsymbol{-1}}\) of the form \eqref{SpeSm} are called
\emph{special} functions. Any function, belonging to a Stieltjes class
(either \(\mathbf{S}\) or \(\mathbf{S}^{\boldsymbol{-1}}\)), that is not special is called
a \emph{generic} function.

The function
\begin{equation}
\label{TE}
\hspace{5.0ex}
\psi(z)=z^{\delta}=e^{\delta(\ln|z|+i\arg{}z)}\,,\
\  |\arg{}z|<\pi \ \textup{ for } \ z\in\mathbb{C}\setminus(-\infty,\,0]\,,
\end{equation}
 is a generic function of the the class \(\mathbf{S}\) if \(-1<\delta<0\) and is
a generic function of the class \(\mathbf{S}^{\boldsymbol{-1}}\) if \(0<\delta<1\).

The function of the form%
\footnote{H.\,Delange, \cite{De}, proved that functions of the form \eqref{MerSF} form a dense subset of the set \(\mathbf{S}\)
with respect to the topology of locally uniform convergence in
 \(\mathbb{C}\setminus(-\infty,0]\).
 The general theory of multiplicative representation of functions from the class \(\mathbf{S}\) is developed in \cite{AD}.}
\begin{subequations}
\label{MerSF}
\begin{gather}
\label{MeSF}
\psi(z)=c\frac{z+b_0}{z+a_0}\prod_{k=1}^{\infty}%
\frac{1+\frac{z}{b_{k}}}{1+\frac{z}{a_{k}}},\\[-2.0ex]
\intertext{where}
\label{CoZ}
c>0,\ \ 0\leq{}a_0<b_0<a_1<b_1<a_2<b_2\,\cdots\,, \ \ \textup{and} \ a_k,\,b_k\to\infty\ \
\textup{as} \ k\to\infty,
\end{gather}
\end{subequations}
is a generic function of the class \(\mathbf{S}\).

The function \(\psi(z)\) belongs to the class \(\mathbf{S}\) if and only if the function \(1/\psi(z)\)
belongs to \(\mathbf{S}^{\boldsymbol{-1}}\).

If \(\psi_1(z),\,\psi_2(z)\) are functions of the class \(\mathbf{S}\)
and \(\alpha_1,\,\alpha_2\) are non-negative constants, \(\alpha_1+\alpha_2>0\),
then the function \(\alpha_1\psi_1(z)+\alpha_2\psi_2(z)\) and
\(\tfrac{\psi_{1}(z)\cdot\psi_{2}(z)}{\psi_{1}(z)+\psi_{2}(z)}\)
  also belongs to the class \(\mathbf{S}\) . The same holds for \(\psi_1(z)\) and \(\psi_2(z)\) with \(\mathbf{S}^{\boldsymbol{-1}}\) in place of \(\mathbf{S}\).

 If \(\psi_1(z),\,\psi_2(z)\) are both functions either of the class \(\mathbf{S}\) or of the class \(\mathbf{S}^{\boldsymbol{-1}}\),
 then their composition \((\psi_1\circ\psi_2)(z)=\psi_1(\psi_2(z))\) is a function of the class \(\mathbf{S^{\boldsymbol{-}1}}\).
 If one of the functions \(\psi_1(z),\,\psi_2(z)\) belongs to the class \(\mathbf{S}\) and the other belongs
 the class \(\mathbf{S}^{\boldsymbol{-1}}\), then  their composition \((\psi_1\circ\psi_2)(z)=\psi_1(\psi_2(z))\)
 is a function of the class  \(\mathbf{S}\).

\begin{remark}\textup{
 The paper \cite{KK} has become a standard reference for anyone writing about the classes \(\mathbf{S}\) and \(\mathbf{S^{\boldsymbol{-}1}}\). Following \cite{KK},
we call \(\mathbf{S}\) and \(\mathbf{S^{-1}}\) the \emph{Stieltjes classes}. (See the last two paragraphs of \cite{KK}.)
It should be mentioned that our definitions of the Stieltjes classes differ slightly from those in \cite{KK}. The functions considered in \cite{KK} are holomorphic in the domain
\(\mathbb{C}\setminus[0,\,+\infty)\). A function \(\psi(z)\) belongs to the class
\(\mathbf{S}\) according to our definition if and only if the
function \(\psi(-z)\) belongs to the class
\(\mathbf{S}\) according to the definition in \cite{KK}. (Similarly, \(\psi(z)\) belongs to our \(\mathbf{S^{-1}}\) if and only if \(-\psi(-z)\) belongs to their \(\mathbf{S^{\boldsymbol{-1}}}\).)
 Results for functions of the Stieltjes classes of our Definition \ref{StCl}
 may be reformulated to agree with \cite{KK} via a change in variables:
 \(\psi(z)\to\pm\psi(-z)\).}
\end{remark}

%%%%%%%%%%%%%%%%%%%%%%%%%%%%%%%%%%%%%%%%%%%%%%%%%%%%%%%%%%%%%%%%%%%%%%%%%
%%%%%%%%%%%%%%%%%%%%%%%%%%%%%%%%%%%%%%%%%%%%%%%%%%%%%%%%%%%%%%%%%%%%%
The following result on integral representations
for functions of the Stieltjes classes
is particularly important for the proof of Theorem \ref{Te1}:
\begin{nonumtheorem} Let \(\psi(z)\) be a function holomorphic in the domain
\(\mathbb{C}\setminus(-\infty,0]\).
\begin{enumerate}
\item[\textup{\textsf{a}}.] \(\psi\) belongs to the class \(\mathbf{S}\)
if and only if \(\psi(z)\) admits the representation
\begin{equation}
\label{CoS}
\psi(z)=a+\frac{b}{z}+\int\limits_{+0}^{+\infty}\frac{1}{\lambda+z}\,d\sigma(\lambda)\,,\ \
z\in\mathbb{C}\setminus(-\infty,0],
\end{equation}
where \(a,\,b\) are non-negative constants, \(d\sigma\) is a non-negative measure on
\((0,\infty)\) satisfying the condition that
\begin{equation}
\label{Cond}
\int\limits_{+0}^{+\infty}\frac{d\sigma(\lambda)}{\lambda+1}<\infty\, ,
\end{equation}
and \(a+b+\int\limits_{+0}^{+\infty}(1+\lambda)^{-1}d\sigma(\lambda)>0\).
\item[\textup{\textsf{b}}.] \(\psi\) belongs to the class \(\mathbf{S^{\boldsymbol-1}}\)
if and only if \(\psi(z)\) admits the representation
\begin{equation}
\label{CoSm}
\psi(z)=a+b{}z+\int\limits_{+0}^{+\infty}\frac{z}{\lambda+z}\,d\sigma(\lambda)\,, \ \
z\in\mathbb{C}\setminus(-\infty,0],
\end{equation}
where \(a\) and \(b\) are non-negative constants, \(d\sigma\) is a non-negative measure on
\((0,\infty)\) satisfying condition \eqref{Cond},
and \(a+b+\int\limits_{+0}^{+\infty}(1+\lambda)^{-1}d\sigma(\lambda)>0\).
 \end{enumerate}
\end{nonumtheorem}
% The condition for the measure \(d\sigma\) which appears in \eqref{CoS} and \eqref{CoSm}
% is:

%%%%%%%%%%%%%%%%%%%%%%%%%%%%%%%%%%%%%%%Armin
Parts \textsf{a.} and \textsf{b.} of the above Theorem appear in \cite{KK} as Theorems
S1.5.1 and S1.5.2, respectively.

\begin{remark}
\label{van} For a function \(\psi\) belonging either to the class \(\mathbf{S}\) or
to the class \(\mathbf{S}^{\boldsymbol{-1}}\), the value \(\psi(0)\) is defined by \eqref{psiaz}.
\begin{enumerate}
\begin{subequations}
\label{pszs}
\item[\textup{1}.] If \(\psi\in\mathbf{S}\), then (using representation \eqref{CoS}) this means
\begin{equation}
\label{pszsp}
\psi(0)=
\begin{cases}
+\infty,& \textup{ if } b>0;\\
a+\int\limits_{+0}^{+\infty}\frac{d\sigma(\lambda)}{\lambda},&  \textup{ if } b=0.
\end{cases}
\end{equation}
\item[\textup{2}.] If \(\psi\in\mathbf{S^{\boldsymbol-1}}\), then (using representation \eqref{CoSm}) this means
\begin{equation}
\label{pszsm}
\psi(0)=a.
\end{equation}
\end{subequations}
\end{enumerate}
\end{remark}
\begin{remark}
\label{CrSp}
\textup{
A function \(\psi\) which belongs to the class \(\mathbf{S}\) is special if and only if the measure
\(d\sigma(\lambda)\) in formula \eqref{CoS} vanishes identically, i.e. iff
\(d\sigma(\lambda)\equiv0\).
(The same holds true for any \(\psi\) in \(\mathbf{S^{\boldsymbol-1}}\)
if and only if \(d\sigma(\lambda)\) in formula \eqref{CoSm} vanishes identically.)}
\end{remark}

\section{A First Approach to Constructing Hurwitz Stable Entire Functions}
\label{Sec3}
\setcounter{equation}{0}

Starting with a function from one of the Stieltjes classes \(\mathbf{S}\)
or \(\mathbf{S}^{\boldsymbol{-}1}\), we will construct a Hurwitz stable
entire function. Our approach is based on the following observation:

\begin{remark}
\label{CoTS}%
Let \(\psi\in\mathbf{S}\).
Then there exist constants \(0<c_1 < c_2<\infty\) such that
\begin{subequations}
\label{tse}
\begin{equation}
\label{tsep}
\dfrac{c_1}{k+1}\leq\psi(k+1)\leq{}c_2, \ \ k=0,\,1,\,2,\,\ldots.
\end{equation}
Let \(\psi\in\mathbf{S}^{-1}\).
Then there exist constants \(0<c_1 < c_2<\infty\) such that
\begin{equation}
\label{tsem}
c_1\leq\psi(k)\leq{}c_2k, \ \ k=1,\,2,\,\ldots\,,
\end{equation}
\end{subequations}
(The constants  which appear in
\eqref{tse}
are independent of \(k\).)
Therefore each of the Taylor series \eqref{DegEp}, \eqref{DegEm} converges and represents an entire function of exponential type one:
\begin{align*}
\varlimsup_{|z|\to\infty}|z|^{-1}\ln|E_{\psi}(z)|&=1
\ \ \textup{for any} \ \ \psi\in \mathbf{S};\\
\varlimsup_{|z|\to\infty}|z|^{-1}\ln|E_{\psi}^{-}(z)|&=1
\ \ \textup{for any} \ \ \psi\in \mathbf{S^{-1}}.
\end{align*}
\end{remark}

\begin{definition} {\ }\\[-2.0ex]
\label{DefgE}
\begin{subequations}
\label{DegEpm}
\begin{enumerate}
\item[\textup{1.}]
Let \(\psi(z)\) be a function from the Stieltjes class \(\mathbf{S}\). The function \(E_{\psi}(z)\) is defined as the sum of the Taylor series
\begin{equation}
\label{DegEp}
E_{\psi}(z)=\sum\limits_{k=0}^{\infty}\frac{\psi(k+1)}{k!}z^{k}.
\end{equation}
\item[\textup{2.}]
Let \(\psi(z)\) be a function from the Stieltjes class \(\mathbf{S}^{\boldsymbol{-}1}\). The function \(E_{\psi}^{-}(z)\) is defined as the sum of the Taylor series
\begin{equation}
\label{DegEm}
E_{\psi}^{-}(z)=\sum\limits_{k=0}^{\infty}\frac{\psi(k)}{k!}z^{k}.
\end{equation}
\end{enumerate}
\end{subequations}
\end{definition}

\begin{theorem} {\ }\\[-2.0ex]
\label{Te1}
\begin{enumerate}
\item[\textup{1.}] Let \(\psi(z)\) be a generic function from
the Stieltjes class \(\mathbf{S}\) and the entire function \(E_{\psi}(z)\) is defined as the sum of the Taylor series
\eqref{DegEp}.\\[0.5ex]
Then the function \(E_{\psi}(z)\) is a Hurwitz stable entire function.
\item[\textup{2.}]
Let \(\psi(z)\) be a generic function from
the Stieltjes class \(\mathbf{S}^{-1}\) and the entire function \(E_{\psi}^{-}(z)\) be defined as the sum of the Taylor series
\eqref{DegEm}.\\[-2.5ex]
\begin{enumerate}
\item
If \(\psi(0)\not=0\), then  the function \(E_{\psi}^{-}(z)\) is a Hurwitz stable entire function.
\item If \(\psi(0)=0\), then  the function \(E_{\psi}^{-}(z)\)
has a simple root at the point \(z=0\) and the function
\(z^{-1}E_{\psi}^{-}(z)\) is a Hurwitz stable entire function.
\end{enumerate}
\end{enumerate}
\end{theorem}
\begin{remark}
\label{mdIf}
Under the assumptions of Theorem \ref{Te1}, the indicator function
\(h_E(\theta)\) of each of the functions \(E_{\psi},\,E_{\psi}^{-}\): \(h_E(\theta)=\varlimsup_{r\to\infty}r^{-1}\ln|E(re^{i\theta})|\),
where \(E\) is either \(E_{\psi}\), or \(E_{\psi}^{-}\), is
 \begin{equation}%
\label{indE}
h_E(\theta)=
\begin{cases}
\cos\theta,&\ \ \textup{if } \ 0\leq|\theta|\leq\pi/2\,,\\[0.5ex]
0,&\ \ \textup{if } \ \pi/2\leq|\theta|\leq\pi\,.
\end{cases}
\end{equation}%
We will prove this later.
\end{remark}
\begin{remark} {\ }\\[-2.5ex]
\label{NoGen}
\begin{subequations}
\label{SpEpm}
\begin{enumerate}
\item[\textup{1.}]
If \(\psi\in\mathbf{S}\) is special,
i.e. \(\psi(z)=a+\dfrac{b}{z}\), where \(a\geq0,\,b>0\),
then the function \(E_{\psi}(z)\) defined as the sum of the Taylor
series \eqref{DegEp} is
\begin{equation}
\label{SpEp}
E_{\psi}(z)=ae^{z}+b\frac{e^{z}-1}{z}.
\end{equation}
The indicator function of \(E_{\psi}\), \eqref{SpEp}, is
of the form \eqref{indE}. If \(a=0\), then all roots of \(E_{\psi}\) are located on the axis \(\textup{Re}\,z=0\) (the boundary of the left half plane). If \(a>0\), then \(E\) has no roots in the closed
right half-plane \(\textup{Re}\,z\geq0\). This function has
infinitely many roots, which
 are located in the open left half-plane \(\textup{Re}\,z<0\) and
 are asymptotically close to the logarithmic "parabola"
 \(x=-\ln|ay/b|\).
 \item[\textup{2.}]
 If \(\psi\in\mathbf{S}^{-1}\) is a special function,
i.e. \(\psi(z)=a+bz\), where \(a\geq0,\,b>0\),
then \(E_{\psi}^{-}(z)\), defined as the sum of the Taylor
series \eqref{DegEm}, is
\begin{equation}
\label{SpEm}
E_{\psi}(z)=(a+bz)e^{z}.
\end{equation}
The indicator function of \(E_{\psi}^{-}\), \eqref{SpEm}, is \(h_{E_{\psi}^{-}}(\theta)=\cos\theta,\,|\theta|~\leq~\pi\).
The function \(E_{\psi}^{-}\), \eqref{SpEm}, has only one root,
which is located at the point \(z~=~-~\frac{a}{b}\).
\end{enumerate}
\end{subequations}
\end{remark}
%%%%%%%%%%%%%%%%%%%%%%%%%%%%%%%%%%%%%%%%%%%%%%%%%%%%%%%%%%%%%%%%%%%%%%%%
%%%%%%%%%%%%%%%%%%%%%%%%%%%%%%%%%%%%%%%%%%%%%%%%%%%%%%%%%%%%%%%%%%%%%%%%%%

%%%%%%%%%%%%%%%%%%%%%%%%%%%%%%%%%%%%%%%%%%%%%%%%%%%%%%%%%%%%%%%%%%%%%%%%%
\begin{remark}
If  \(\psi(z)\in\mathbf{S}\), it is natural to consider functions
of the form \(\sum\limits_{k=1}^{\infty}\frac{\psi(k)}{k!}z^k\),
 rather than of the form \(\sum\limits_{k=1}^{\infty}\frac{\psi(k+1)}{k!}z^k\),
and ask whether zeros of such functions are located in the left half-plane. The answer to this question is, in  general, that there are \emph{not}. The function
\begin{equation}
\label{GeBa}
E_{\delta,a}(z)=\sum\limits_{k=0}^{\infty}\frac{(k+a)^{\delta}z^k}{k!}
\end{equation}
is discussed in \cite{O} and there it is shown that, for \(0<a<1,
-1<\delta<0\), all roots of the function \(E_{\delta,a}(z)\) are located in
the open \emph{right} half-plane. However, the function
\eqref{GeBa} is of the form \(\sum\limits_{k=1}^{\infty}\frac{\psi(k)}{k!}z^k\) with
\(\psi(z)=(z+a)^{\delta}\). If \(-1<\delta<0\) and \(a>0\), then this function
\(\psi\) belongs to the Stieltjes class \(\mathbf{S}\) and is
holomorphic at the point \(z=0\).
\end{remark}

%%%%%

In preparation for our proof of Theorem \ref{Te1}, we
will next establish a few useful results.

\begin{lemma}
\label{FRL}%
 Let \(d\sigma(\lambda)\) be a non-negative measure on
\((0,\,\infty)\) which satisfies condition \eqref{Cond} and
such that \(d\sigma(\lambda)\not\equiv0\), i.e. for which
\begin{math}
%\label{NoDeg}
\int\limits_{+0}^{\infty}\frac{d\sigma(\lambda)}{1+\lambda}>0\,.
\end{math}
Furthermore, let
\begin{equation}
\label{AMT}
\varphi_{\sigma}(t)=\int\limits_{+0}^{\infty}t^{\lambda}\,d\sigma(\lambda)\,,\ \ 0<{}t<1\,.
\end{equation}
Then
\begin{enumerate}
\item[\textup{\textsf{a}}.] \ \ \( 0<\varphi_{\sigma}(t)<\infty\ \text{ for }\ 0<t<1\).
\item[\textup{\textsf{b}}.]
The function \(\varphi_{\sigma}(t)\) is strictly increasing on
\((0,1)\), i.e.
\begin{equation*}
%\label{SIn}
\varphi_{\sigma}(t^{\prime})<\varphi_{\sigma}(t^{\prime\prime}) \ \textup{ for }\ 0<t^{\prime}<t^{\prime\prime}<1\,.
\end{equation*}
\item[\textup{\textsf{c}}.] The function \(\varphi_{\sigma}(t)\) is summable on \((0,1)\), i.e.
\begin{math}
\displaystyle \int\limits_{0}^{1 }\varphi_{\sigma}(t)\,dt<\infty\,.
\end{math}
\end{enumerate}
\end{lemma}
\begin{proof}
Parts \textsf{a} and \textsf{b} are clear. To prove \textsf{c}, we substitute expression \eqref{AMT} for \(\varphi(t)\) in
\(\int\limits_{0}^{1}\varphi_{\sigma}(t)dt\) and change the order of integration:
\[\int\limits_{0}^{\infty}\varphi_{\sigma}(t)\,dt=\int\limits_{+0}^{\infty}\frac{d\sigma(\lambda)}{\lambda+1}\,.\]
\end{proof}
\begin{remark}
\label{GK}%
Under the assumptions of \textup{Lemma} \ref{FRL}, the kernel
\begin{equation}%
\label{Ker}
K_{\sigma}(t,s)=\varphi_{\sigma}(t\cdot{}s),\,0<t,\,s<1\,,
\end{equation}
 is positive definite, i.e.
\begin{equation*}%
\sum\limits_{1\leq{}p,\,q\leq{}n}K_{\sigma}(t_p,t_q)\,\xi_p\,\xi_q\geq0
\end{equation*}
for any \(n\in\mathbb{N}\), any \(t_1,\,\ldots\,,t_n\in(0,\,1)\)
and any \(\xi_1,\,\ldots{},\,\xi_n\in\mathbb{C}\).
\end{remark}
%%%%%%%%%%%%%%%%%%%%%%%%%%%%%%%%%%%%%%%%%%%%%%%%%%%%%%%%%%%%%%%%%%%%%%%%%
\begin{lemma} {\ }\\[-1.5ex]
\label{ERepr}
\begin{enumerate}
\item[\textup{\textsf{1}}.] Let \(\psi(z)\) be a function from \(\mathbf{S}\) and the entire function \(E_{\psi}(z)\)
    be defined as the sum of the Taylor series \eqref{DegEp}.\\[0.5ex]
The function \(E_{\psi}(z)\) then admits the integral representation
\begin{equation}
\label{InReEp}
E_{\psi}(z)=ae^{z}+\int\limits_{0}^{1}(b+\varphi_{\sigma}(t))e^{tz}\,dt\,,
\end{equation}
where \(a\) and \(b\)  are the constants and \(d\sigma\) is the measure which appears
in the representation of \(\psi\) in formula \eqref{CoS}; the function \(\varphi_{\sigma}(t)\) is
constructed from the measure \(d\sigma\) according to \eqref{AMT}.
\item[\textup{\textsf{2}}.]
Let \(\psi(z)\) be a function from \(\mathbf{S^{-1}}\) and the entire function \(E_{\psi}^{-}(z)\)
    be defined as the sum of the Taylor series \eqref{DegEm}.\\[0.5ex]
The function \(E_{\psi}^{-}(z)\) then admits the integral representation
\begin{equation}
\label{InReEm}
E_{\psi}^{-}(z)=(a+bz)e^{z}+z\int\limits_{0}^{1}\varphi_{\sigma}(t)e^{tz}\,dt\,,
\end{equation}
where \(a\) and \(b\)  are the constants and \(d\sigma\) is the measure which appears
in the representation of \(\psi\) in formula \eqref{CoSm}; the function \(\varphi_{\sigma}(t)\) is
constructed from the measure \(d\sigma\) according to \eqref{AMT}.
\end{enumerate}
\end{lemma}
%%%%%%%%%%%%%%%%%%%%%%%%%%%%%%%%%%%%%%%%%%%%%%%%%%%%%%%%%%%%%%%%%%%%%
\begin{proof}[Proof of Lemma \ref{ERepr}]
Let \(\varphi_{\sigma}(t)\) be given by \eqref{AMT}. Furthermore, let
\begin{equation}%
\label{Dfsi}
\tilde{E}_{\psi}(z)\stackrel{\textup{\tiny def}}{=}ae^{z}+\int\limits_{0}^{1}e^{tz}(b+\varphi_{\sigma}(t))\,dt\,.
\end{equation}
We now use the Taylor expansion
\begin{math}
e^{tz}=\sum\limits_{k=0}^{\infty}\frac{1}{k!}t^{k}z^{k}
\end{math}
and \eqref{AMT} to evaluate \eqref{Dfsi}.

% the representation \eqref{AMT} for \(\varphi_{\sigma}(t)\)
% and  into the expression

Changing the order of integration and summation, we obtain
\begin{multline*}%
\tilde{E}_{\psi}(z)=a\,\sum\limits_{k=0}^{\infty}\frac{1}{k!}z^k
+\int\limits_{0}^{1}%
\bigg(\sum\limits_{k=0}^{\infty}\frac{1}{k!}t^kz^k\bigg)
\bigg(b+\int\limits_{0}^{\infty}t^{\lambda}%
\,d\sigma(\lambda)\bigg)\,dt=\\
=\sum\limits_{k=0}^{\infty}%
\bigg(a+\frac{b}{k+1}+\int\limits_{0}^{\infty}
\Big(\int\limits_{0}^{1}t^{\lambda+k}\,dt
\Big)d\sigma(\lambda)\bigg)\frac{1}{k!}z^k=\\
=\sum\limits_{k=0}^{\infty}\bigg(a+\frac{b}{k+1}+
\int\limits_{0}^{\infty}
\frac{d\sigma(\lambda}{\lambda+k+1}\bigg)\frac{1}{k!}z^k=
\sum\limits_{k=0}^{\infty}\frac{\psi(k+1)}{k!}z^k\,.
\end{multline*}
Thus, \(\tilde{E}_{\psi}(z)=E(z)\). The first part of the Lemma is therefore proved.
Similarly, we obtain the second part of the Lemma.
\end{proof}
%%%%%%%%%%%%%%%%%%%%%%%%%%%%%%%%%%%%%%%%%%%%%%%%%%%%%%%%%%%%%%%%%%%%
%%%%%%%%%%%%%%%%%%%%%%%%%%%%%%%%%%%%%%%%%%%%%%%%%%%%%%%%%%%%%%%%%%%%
%%%%%%%%%%%%%%%%%%%%%%%%%%%%%%%%%%%%%%%%%%%%%%%%%%%%%%
\begin{lemma}
\label{MaL}
 Let \(u(t)\) be a summable function on \((0,\infty)\) which satisfies the conditions:\\[-1.5ex]
    \begin{enumerate}
    \item[\textup{1.}] \(u\) is decreasing, i.e. \(u(t^{\prime})\geq{}u(t^{\prime\prime})\),
     if \(0<t^{\prime}<t^{\prime\prime}<\infty\).
     \item[\textup{2.}] \(u\) decreases strictly for small
     \(t\), i.e. \,%
     \(\exists\,\varepsilon>0:\,%
     u(t^{\prime})>u(t^{\prime\prime})\) \,
     if \,\(0<t^{\prime}<t^{\prime\prime}<\varepsilon\).
     \item[\textup{3.}] \(\lim_{t\to+\infty}u(t)=0\).
    \end{enumerate}
    Then, for every \(y>0\):
    \begin{equation}
   \label{CrIn}
    \int\limits_{0}^{+\infty}u(t)\sin(yt)\,dt>0.
    \end{equation}
\end{lemma}
\begin{proof}

 The integral in \eqref{CrIn} converges according to the Abel Criterion.
 The contents of Lemma \ref{MaL} might appear familiar,
 however their explicit formulation is not easy to find in the literature.
 For completeness, we offer a proof.
 (The main idea of the proof can be found in \cite[\S{}3.3, Theorem 34]{HR}.
 See also \cite[Lemma 4]{OP}.)
We transform the integral \eqref{CrIn} in such a way that its
positivity becomes apparent:
\begin{multline*}
\int\limits_{0}^{+\infty}u(t)\sin(yt)\,dt=
\frac{1}{y}\int\limits_{0}^{+\infty}u(t/y)\sin(t)\,dt=\frac{1}{y}\sum\limits_{j=0}^{\infty}
\bigg(\int\limits_{j\pi}^{(j+1)\pi}u(t/y)\sin t\,dt\bigg)=\\
=\frac{1}{y}\sum\limits_{j=0}^{\infty}
\bigg(\int\limits_{0}^{\pi}u((t+j\pi)/y)\sin(t+j\pi)\,dt\bigg)=\\
=\frac{1}{y}\int\limits_{0}^{\pi} \bigg(\sum\limits_{j=0}^{\infty}
 u\Big(\frac{t+2j\pi}{y}\Big)-u\Big(\frac{t+(2j+1)\pi}{y}\Big)\bigg)\,\sin{}t\,dt\,.
\end{multline*}
Since \(u\) decreases, each term
\(u\big((t+2j\pi)/y\big)-u\big((t+(2j+1)\pi)/y\big)\) of the above
series is non-negative for \(t>0\). Since \(u(t)\) decreases
strictly for small \(t\), the term with \(j=0\) is strictly positive if \(t>0\) is
small enough. Therefore, the strict inequality \eqref{CrIn} holds.
\end{proof}
%%%%%%%%%%%%%%%%%%%%%%%%%%%%%%%%%%%%%%%%%%%%%%%%%%%%%%%%%%%%%%%%%%%%%%%
%%%%%%%%%%%%%%%%%%%%%%%%%%%%%%%%%%%%%%%%%%%%%%%%%%%%%%%%%%%%%%%%%%%%%%
\begin{proof}[Proof of Theorem \ref{Te1}]
To prove that the entire function \(E_{\psi}(z)\) is stable,
we have to prove that \(h_{E_{\psi}}(\pi)\leq{}h_{E_{\psi}}(0)\)
and that \(E_{\psi}(z)\) has no roots in the right half-plane.
We obtain these results not from the Taylor series \eqref{DegEp}, but
from the integral representation \eqref{InReEp}.\\
\textbf{1.} Since coefficients of the Taylor series \eqref{DegEp} are positive,
the inequality
\begin{equation}
\label{CoTspm}
|E_{\psi}(-r)|<|E_{\psi}(r)|, \ \ 0<r<\infty
\end{equation}
holds. All the more, the inequality
\begin{equation}
\label{PosDE}
h_{E_{\psi}}(\pi)\leq h_{E_{\psi}}(0)
\end{equation}
holds. Nevertheless it may be interesting to calculate the indicator
function \(h_{E_{\psi}}(\theta)\) for all \(\theta\).
 From
\eqref{InReEp} it follows that
\[
|E_{\psi}(re^{i\theta})|\leq{}ae^{r\cos\theta}+
\int\limits_{0}^{1}(b+\varphi_{\sigma}(t))
\max\limits_{t\in[0,1]}(e^{tr\cos\theta})\,dt,\ \  (0\leq{}r<\infty).
\]
It is clear that
\[
\max\limits_{t\in[0,1]}(e^{tr\cos\theta})=
\begin{cases}
e^{r\cos\theta}&, \ \textup{if} \ 0\leq|\theta|\leq\frac{\pi}{2},\\
1&, \ \textup{if} \ \frac{\pi}{2}\leq|\theta|\leq\pi.
\end{cases}
\]
Thus,
\begin{gather*}
|E_{\psi}(re^{i\theta})|
\leq(a+b+\int\limits_{0}^{1}\varphi_{\sigma}(t)dt)\,e^{r\cos\theta}
\ \ \textup{for} \ \ \ 0\leq|\theta|\leq\frac{\pi}{2},\\
|E_{\psi}(re^{i\theta})|
\leq{}a+b+\int\limits_{0}^{1}\varphi_{\sigma}(t)dt
\ \ \textup{for} \ \ \ \frac{\pi}{2}\leq|\theta|\leq\pi\,.
\end{gather*}
Hence,
\begin{equation}
\label{uein}
h_{E_{\psi}}(\theta)\leq
\begin{cases}
\cos\theta&, \ \textup{if} \ 0\leq|\theta|\leq\frac{\pi}{2},\\
0&, \ \textup{if} \ \frac{\pi}{2}\leq|\theta|\leq\pi.
\end{cases}
\end{equation}
On the other hand, since
\begin{equation*}
e^{re^{i\theta}}\geq0\ \ \textup{for} \ \ \theta=0 \textup{ as well as for } \theta=\pi,
\end{equation*}
the inequality
\begin{equation}
\label{efb}
E_{\psi}(re^{i\theta})\geq
\int\limits_{0}^{1}\varphi_{\sigma}(t)e^{tr\cos\theta}\,dt
\end{equation}
holds for \(\theta=0\) and \(\theta=\pi\). Because the function
\(\varphi_{\sigma}(t)\)
is \emph{strictly} positive for \(t\in(0,1)\), it follows from
\eqref{efb} that
\begin{subequations}
\label{IhEprzpi}
\begin{gather}
\label{IhEprz}
h_{E_{\psi}}(0)\geq1, \\
\label{IhEprpi}
 h_{E_{\psi}}(\pi)\geq0.
\end{gather}
\end{subequations}
From \eqref{uein}, the converse inequalities
\begin{equation*}
h_{E_{\psi}}(0)\leq1, \ \ \ h_{E_{\psi}}(\pi)\leq0
\end{equation*}
follow. Thus, we have proved that
\begin{equation}
\label{fiin}
h_{E_{\psi}}(0)=1, \ \ \ h_{E_{\psi}}(\pi)=0.
\end{equation}
\emph{In particular, the inequality
\eqref{PosDE} holds.}
From \eqref{uein} it follows that
\begin{equation*}
h_{E_{\psi}}(\pi/2)\leq0, \ \ \ h_{E_{\psi}}(-\pi/2)\leq0.
\end{equation*}
The indicator function of an entire function of exponential type
is a support function of some compact convex set. The inequality
\begin{equation*}
h_{E_{\psi}}(\pi/2)+ h_{E_{\psi}}(-\pi/2)\geq0
\end{equation*}
expresses the fact that the width of the convex set
related to \(E_{\psi}\) is non-negative in the
direction \(\theta=\pm\frac{\pi}{2}\). Thus,
\begin{equation}
\label{fiinv}
h_{E_{\psi}}\big(\tfrac{\pi}{2}\big)=0, \ \ \ h_{E_{\psi}}\big(\!-\tfrac{\pi}{2}\big)=0,
\end{equation}
and the appropriate convex set is a closed subinterval of the real axis. From the equalities \eqref{fiin} it follows that this subinterval is the interval \([0,1]\), and the indicator function \(h_{E_{\psi}}\)
is of the form \eqref{indE}.

\textbf{2.}\ We now prove that \(E_{\psi}\) has no roots in
the closed right half-plane.
From \eqref{InReEp} it follows that
the imaginary part of \(e^{-z}E_{\psi}(z)\) can be represented as
\begin{equation}
\label{Im}
\textup{Im}\,\big(e^{-z}E_{\psi}(z)\big)=-\int\limits_{0}^{\infty}u(t)\sin{}(yt)\,dt\,,\ \ (z=x+iy)\,,
\end{equation}
where
\begin{equation}%
\label{U}%
 u(t)=
\begin{cases}
e^{-xt}\,(b+\varphi_{\sigma}(1-t)),&\ \textup{ if } \ 0<t<1\,;\\
0,&\ \textup{ if } \ 1\leq{}t<\infty\,.
\end{cases}
\end{equation}
For fixed \(x\geq0\), the function \(u(t)\) is monotone decreasing for \(t\in(0,\,\infty)\)
and \(u(t)=0\) for \(t\geq1\).
For \(t\in(0,1)\) we see, moreover, that \(u(t)\) is strictly monotone decreasing.
According to Lemma \ref{MaL}, \(\textup{Im}\,\big(e^{-z}E(z)\big)<0\) for \(x\geq{}0,\,y>0\).
Furthermore, \(E_{\psi}(z)\not=0\) for \(x\geq{}0,\,y>0\),\,(\(z=x+iy\)). Since the function \(E_{\psi}\) is real, \(E_{\psi}(z)\not=0\) for \(x\geq{}0,\,y<0\). Since all Taylor coefficients of \(E_{\psi}\) are positive,
\(E_{\psi}(x)>0\) for \(x\geq{}0\). Thus, \(E_{\psi}(z)\) has no roots in the closed
right half-plane
\(\textup{Im}\,z\geq{}0\). The first part of Theorem \ref{Te1} is thus proved.

Similarly, by using
\begin{equation}
\label{Imm}
\textup{Im}\,\big(z^{-1}e^{-z}E_{\psi}(z)\big)=-a\frac{y}{x^2+y^2}
-\int\limits_{0}^{\infty}u(t)\sin{}(yt)\,dt\,,\ \ (z=x+iy)\,,
\end{equation}
where
\begin{equation}%
\label{Um}%
 u(t)=
\begin{cases}
e^{-xt}\,\varphi_{\sigma}(1-t),&\ \textup{ if } \ 0<t<1\,;\\
0,&\ \textup{ if } \ 1\leq{}t<\infty\,,
\end{cases}
\end{equation}
instead of \eqref{Im},
we obtain the second part of the theorem.
\end{proof}

%%%%%%%%%%%%%%%%%%%%%%%%%%%%%%%%%%%%%%%%%%%%%%%%%%%%%%%%%%%%%%%%%%%%%
%%%%%%%%%%%%%%%%%%%%%%%%%%%%%%%%%%%%%%%%%%%%%%%%%%%%%%%%%%%%%%%%%%%%%
\begin{remark}
From the representation of the
function \(E_{\psi}(z)\) in \eqref{InReEp}, it follows that \(E_{\psi}(z)\) is bounded in the left half plane, i.e.
\begin{math}
|E_{\psi}(z)|\leq{}C<\infty
\end{math}
for \(\textup{Re}\,z\leq{}0\). From this estimate it follows that the
roots \(\{z_k\}\) of the function \(E_{\psi}(z)\) satisfy the Blaschke condition
\begin{equation*}
%\label{BC}%
\sum\limits_{\forall\,z_k}\frac{|\textup{Re}\,z_k|}{1+|z_k|^2}<\infty\,.
\end{equation*}
Furthermore, because the indicator function \(h_{E_{\psi}}(\theta)\) is of the form
\eqref{indE}, the function \(E_{\psi}\) has infinitely many roots. These roots have
a positive density and are concentrated near the boundary \(\textup{Re}\,z=0\) of the left
half-plane.    More precisely, for \(r>0\) and \(\alpha<\beta\), let
\(n(r;\,\alpha,\,\beta)\)  be the total number of roots of~\(E\) which are contained
in the sector \(\{z:\,0\leq|z|<r,\,\alpha<\arg{}z<\beta\}\).

  Then, for any~\(\varepsilon>0\):
\begin{equation}
\label{DenE}
   \lim_{r\to\infty}\frac{n(r;\,0,\varepsilon)}{r}=\frac{1}{\pi}\,,
   \ \
   \lim_{r\to\infty}\frac{n(r;\,\pi-\varepsilon,\pi)}{r}=\frac{1}{\pi}\,,\ \
     \lim_{r\to\infty}\frac{n(r;\,\varepsilon,\pi-\varepsilon)}{r}=0\,.
\end{equation}
The analogous results hold for the functions of the form \(E_{\psi}^{-}\).
\end{remark}

\section{A Further Approach to Constructing Hurwitz Stable Entire Functions}
\label{Sec4}

Before we can discuss another approach to constructing Hurwitz stable
entire functions, we must first, following P\'{o}lya and Schur \cite{PS},
introduce a class of entire functions that will suit our purposes.

% Following P\'{o}lya and Schur \cite{PS}, we introduce the following definition.
\begin{definition}
\label{LP1D}
An entire rational or transcendent function \(F(z)\) is called an \emph{entire function of type I}
if is representable in the form
\begin{equation}
\label{CaR}
F(z)=Cz^{m}\,e^{\alpha{}z}\prod\limits_{\nu=1}^{\infty}(1+\delta_{\nu}z),\ \ (C\gtrless0,\,m\geq0,\,\alpha\geq0,\,\delta_{\nu}\geq0),
\end{equation}
where
\begin{math}
\sum\limits_{\nu=0}^{\infty}\delta_{\nu}<\infty\,.
\end{math}\\
\vspace{1.5ex}
\noindent
The class of all entire functions of type \textup{I} will be denoted by \(\mathscr{LP}\)-\textup{I}.
\end{definition}

It was shown by Laguerre and P\'{o}lya \cite{Po1} that
the entire functions of type I \emph{and no others} are
uniform limits (\(\not\equiv0\)) of real polynomials having all their roots on the negative half-axis \(x\leq0\) (\(z=x+iy\)).
(Laguerre assumes uniform convergence in every finite domain. P\'{o}lya assumes this convergence only in a neighborhood of the origin.)

If \(F(z)\) is an entire function of the form \eqref{CaR}, then
its growth is not more than exponential, i.e.
\begin{equation}
\label{ETF}
\varlimsup_{z\to\infty}|z|^{-1}\ln|F(z)|=\alpha,
\end{equation}
and its indicator function
 \begin{equation}
 \label{Indf}
 h_F(\theta)=\varlimsup_{r\to\infty}r^{-1}\,\ln|F(re^{i\theta})|
 \end{equation}
  is
\begin{equation}%
\label{indf}
h_F(\theta)=
\alpha \cos\theta, \ \ \ \ 0\leq|\theta|\leq\pi\,.\\[0.5ex]
\end{equation}%
Furthermore, the limit (not only the upper limit) in \eqref{Indf} exists for \(\theta\not=\pi\).
\begin{definition} {\ }
\label{DefgF}
Let \(F(z)\) be an entire function from the Laguerre-P\'{o}lya class
\(\mathscr{LP}\)-\textup{I}
with Taylor series
\begin{equation}
\label{TAY}
F(z)=\sum\limits_{k=0}^{\infty}f_k{z^k}\,.
\end{equation}
\begin{subequations}
\label{DegFpm}
\begin{enumerate}
\item[\textup{1.}]
 Given  a function \(\psi(z)\) from the Stieltjes class \(\mathbf{S}\), the function \(F_{\psi}(z)\) is defined as the sum of the Taylor series
\begin{equation}
\label{DegFp}
F_{\psi}(z)=\sum\limits_{k=0}^{\infty}\psi(k+1)f_{k}z^{k}.
\end{equation}
\item[\textup{2.}]
 Given a function \(\psi(z)\) from the Stieltjes class \(\mathbf{S}^{\boldsymbol{-1}}\), the function \(F_{\psi}^{-}(z)\) is defined as the sum of the Taylor series
\begin{equation}
\label{DegFm}
F_{\psi}^{-}(z)=\sum\limits_{k=0}^{\infty}\psi(k)f_{k}z^{k}.
\end{equation}
\end{enumerate}
\end{subequations}
\end{definition}

From the Cauchy estimates for the Taylor coefficients \(f_k\) of the function \(F(z)\), \eqref{CaR}, and from the estimates \eqref{tse}, it follows that,
for any \(\psi\in\mathbf{S}\),
the Taylor series \eqref{DegFp} converges
and that its limit is a function of exponential type, \(\alpha\):
\begin{subequations}
\label{Grfpm}
\begin{equation}
\label{Grfp}
\varlimsup_{z\to\infty}|z|^{-1}\ln|F_{\psi}(z)|=\alpha.
\end{equation}
Similarly, for any \(\psi\in\mathbf{S^{\boldsymbol-1}}\),
the Taylor series \eqref{DegFm}  converges
and its limit is also a function of  type, \(\alpha\):
\begin{equation}
\label{Grfm}
\varlimsup_{z\to\infty}|z|^{-1}\ln|F_{\psi}^{-}(z)|=\alpha.
\end{equation}
\end{subequations}
The \(\alpha\) in \eqref{Grfp} and \eqref{Grfm} is the same one as in
formulas \eqref{CaR} and \eqref{ETF}.
\begin{theorem} {\ }\\
\label{Te2}
 Let \(F(z)\) be an entire function from the Laguerre-P\'{o}lya class \mbox{\textup{\(\mathscr{LP}\)-I}},  \(F(0)\not=0\) and \eqref{TAY} be
 the Taylor series of the function \(F\).
\begin{enumerate}
\item[\textup{1.}]
Given a generic function \(\psi(z)\) from
the Stieltjes class \(\mathbf{S}\),
 let the function \(F_{\psi}(z)\) be defined as the sum of the Taylor series \eqref{DegFp}.\\
 \hspace*{2.5ex}Then \(F_{\psi}(z)\) is a Hurwitz stable entire function.
\item[\textup{2.}]
Given a generic function \(\psi(z)\) from
the Stieltjes class \(\mathbf{S^{\boldsymbol{-}1}}\),
 let the function \(F_{\psi}^{-}(z)\) be defined as the sum of the Taylor series \eqref{DegFm}.\\[-2.5ex]
\begin{enumerate}
\item
If \(\psi(0)\not=0\), then  \(F_{\psi}^{-}(z)\) is a Hurwitz stable entire function.
\item If \(\psi(0)=0\), then \(F_{\psi}^{-}(z)\)
has a simple root at \(z=0\) and
\(z^{-1}F_{\psi}^{-}(z)\) is a Hurwitz stable entire function.
\end{enumerate}
\end{enumerate}
\end{theorem}

Theorem \ref{Te1} is a special case of Theorem \ref{Te2} corresponding to the choice \(F(z)=e^{z}\). Theorem \ref{Te1}, however, is used in our proof of Theorem \ref{Te2}.

The following Lemma will be used to prove Theorem \ref{Te2} in much the same way that Lemma \ref{ERepr} was used to prove Theorem \ref{Te1}.
\begin{lemma}
\label{CrReprF}
 Let \(F(z)\) be an entire function from the Laguerre-P\'{o}lya class \mbox{\textup{\(\mathscr{LP}\)-I}}      and suppose that \eqref{TAY} is
 the Taylor series of the function \(F\).
\begin{enumerate}
\item[\textup{\textsf{1}}.]Given a function \(\psi(z)\) from
the Stieltjes class \(\mathbf{S}\),
 let the function \(F_{\psi}(z)\) be defined as the sum of the Taylor series \eqref{DegFp}.\\
 \hspace*{2.5ex}Then \(F_{\psi}(z)\) admits the integral representation
 \begin{equation}
\label{InReF}
F_{\psi}(z)=aF(z)+\int\limits_{0}^{1}(b+\varphi_{\sigma}(t))F(tz)\,dt\,,
\end{equation}
where \(a\) and \(b\)  are the constants and \(d\sigma\) is the measure which appears
as part of the representation of \(\psi\) in \eqref{CoS}; the function \(\varphi_{\sigma}(t)\) is
constructed from the measure \(d\sigma\) according to \eqref{AMT}.
\item[\textup{\textsf{2}}.]
Given a generic function \(\psi(z)\) from
the Stieltjes class \(\mathbf{S^{\boldsymbol{-}1}}\),
 let the function \(F_{\psi}^{-}(z)\) be defined as the sum of the Taylor series \eqref{DegFm}.\\
 \hspace*{2.5ex}Then \(F_{\psi}^{-}(z)\) admits the integral representation
\begin{equation}
\label{InReFm}
F_{\psi}^{-}(z)=(a+bz)F^{\prime}(z)+
z\int\limits_{0}^{1}\varphi_{\sigma}(t)F^{\prime}(tz)\,dt\,,
\end{equation}
where \(a\) and \(b\)  are the constants and \(d\sigma\) is the measure which appears
as part of the representation of \(\psi\) in \eqref{CoSm}; the function \(\varphi_{\sigma}(t)\) is
constructed from the measure \(d\sigma\) according to \eqref{AMT}.
\end{enumerate}
\end{lemma}
\begin{proof} The proof of Lemma \ref{CrReprF} is, in almost all ways, similar to the proof of Lemma \ref{ERepr}. The sole difference is that \(F(z)\)'s Taylor expansion \eqref{TAY} is used instead of the Taylor expansion \(e^z=\sum\limits_{k=0}^{\infty}\frac{1}{k!}z^k\) of the exponential function.
\end{proof}

%%%%%%%%%%%%%%%%%%%%%%%%%%%%%%%%%%%%%%%%%%%%%%%%%%%%%%%%%%%%%%%%%%%%%%%%%%
\begin{proof}[Proof of Theorem \ref{Te2}] {\ }\\
We already know that the functions \(F_{\psi},\,F_{\psi}^{-}\) are
entire functions of exponential type~\(\alpha\). (See \eqref{Grfpm}.)
To prove the first part of the theorem, we have to prove that
\begin{equation}
\label{PosDeFp}
h_{F_{\psi}}(0)\geq{}h_{F_{\psi}}(\pi)
\end{equation}
and that all zeros of the function \(F_{\psi}\) lie in the left half-plane.\\
\textbf{1.} Inequality \eqref{PosDeFp} can be proved in much the same way
as inequality \eqref{PosDE}. Because of \eqref{indf}, we obtain for every
\(\varepsilon>0\) and \(r>0\), the following estimate for \(F(re^{i\theta})\):
\begin{equation*}
|F(re^{i\theta})|\leq{}C(\varepsilon)e^{\alpha{}r\cos\theta+\varepsilon{}r},
\end{equation*}
where \(C(\varepsilon)\) does not depend on  \(r\), \(C(\varepsilon)<\infty\) for every
\(\varepsilon>0\). Combining this inequality with the integral representation
\eqref{InReF}, we obtain the inequality for the indicator function
\(h_{F_{\psi}}(\theta)\):
\begin{equation}%
\label{indFp}
h_{F_{\psi}}(\theta)\leq
\begin{cases}
\alpha\cos\theta,&\ \ \textup{if } \ 0\leq|\theta|\leq\pi/2\,,\\[0.5ex]
0,&\ \ \textup{if } \ \pi/2\leq|\theta|\leq\pi\,.
\end{cases}
\end{equation}%
We obtain inequality \eqref{indFp} in the same way as we did inequality \eqref{uein}. In particular, from \eqref{indFp} it follows that
\begin{subequations}
\label{IhFprzpi}
\begin{gather}
\label{IhFprz}
h_{F_{\psi}}(0)\leq{}\alpha,\\
\label{IhFprpi}
h_{F_{\psi}}(\pi)\leq0.
\end{gather}
\end{subequations}

We obtain the converse inequality to \eqref{IhFprz} just as we did the inequality \eqref{IhEprz}. Without loss of generality,
we can assume that \(F(0)>0\). From representation \eqref{LP1D}, we get
the inequality
\begin{equation}
\label{InFrpo}
F(r)\geq{}F(0)e^{\alpha{}r}, \ \ 0\leq{}r<\infty.
\end{equation}
In particular, \(F(r)\) is positive for \(r\geq0\).
Combining \eqref{InFrpo} with representation \eqref{InReF}, we obtain the inequality
\begin{equation}
\label{EFpfb}
F_{\psi}(r)\geq{}aF(0)e^{\alpha{}r}+
F(0)\int\limits_{0}^{1}(b+\varphi_{\sigma}(t))e^{\alpha{}tr}dt,\ \ 0\leq{}r<\infty.
\end{equation}
From this inequality and the properties of the function \(\varphi_{\sigma}(t)\)
(See part \textsf{a.} of Lemma \ref{FRL}), we obtain the
the estimate,
\begin{math}
h_{F_{\psi}}(0)\geq\alpha.
\end{math}
The converse inequality \eqref{IhFprz} has already been shown.
Thus,
\begin{equation}
\label{viFaz}
h_{F_{\psi}}=\alpha.
\end{equation}
Inequality \eqref{PosDeFp} follows from equality \eqref{viFaz}
and inequality \eqref{IhFprpi}.

The converse inequality to \eqref{IhFprpi} does not, in general, hold. The difference between the cases of the functions \(E_{\psi}\)
and \(F_{\psi}\) is that the exponential function \(e^{-r}\) is positive
for all \(r\geq0\), but the function \(F(-r)\) takes both positive and
negative values for \(0\leq{}r<\infty\). Thus, the value \(h{F_{\psi}}\)
is non-positive in general:
\begin{equation}
\label{viFapi}
h_{F_{\psi}}(\pi)=-\beta,\ \ \textup{where} \ \ 0\leq\beta\leq\alpha.
\end{equation}
(The inequality \(\beta\leq\alpha\) means that \(h_{F_{\psi}}(0)+h_{F_{\psi}}(\pi)\geq0\).)

From \eqref{viFaz}, \eqref{viFapi} and \eqref{indFp}, it follows that the
indicator function \(h_{F_{\psi}}(\theta)\) is the support function
of the interval \([\beta,\alpha]\) of the real axis:
\begin{equation*}
h_{F_{\psi}}(\theta)=
\begin{cases}
\alpha\cos\theta, & \ \ \textup{if} \ \ |\theta|\leq\pi/2,\\
\beta\cos\theta, & \ \ \textup{if} \ \ \pi/2\leq|\theta|\leq\pi,
\end{cases}
\end{equation*}
where \(\alpha\) is the same as in \eqref{LP1D} and \(\beta,\,0\leq\beta\leq\alpha\) depends on \(F\) and \(\psi\).\\[0.5ex]

\noindent
\textbf{2.} According to Theorem \ref{Te1}, \(E_{\psi}(z)\) is an entire function of exponential
type. Its defect \(d_{E_\psi}=\dfrac{h_{E_\psi}(0)-h_{E_\psi}(\pi)}{2}\) is positive and all roots of
\(E_{\psi}\) lie in the closed left half-plane \(\textup{Re}\,z\leq0\).

\emph{From these properties of the function \(E_{\psi}\) it follows that there exists a sequence of polynomials \(\{E_n(z)\}_{1\leq{}n<\infty}\),
\begin{equation}
\label{PoP}
E_n(z)=\sum\limits_{k}\frac{e_{k,n}}{k!}z^k,
\end{equation}
possessing the properties}:
\begin{enumerate}
\item[1.]\emph{ The sequence \(\{E_n(z)\}_{1\leq{}n<\infty}\) converges
to the function \(E_{\psi}(z)\) locally uniformly in the complex plane
\(\mathbb{C}\). \textup{(}In particular, \(\lim_{n\to\infty}e_{k,n}=\psi(k+1)\) for each \(k\).\textup{)}
\item[2.] For every \(n\), the roots of the polynomial \(E_n(z)\) lie
in the closed left half-plane \(\textup{Re}\,z\leq0\).}
\end{enumerate}
These results can be found in \cite{L1} (Combine Theorem 4 of Chapt.\,8,\,
Lemma\,1 of  Chapt.\,7
and Theorem 6 of Chapt.\,7 from the book \cite{L1}.)

Let the polynomial \(F_n(z)\) be defined as the multiplicative composition
\begin{equation}
\label{PoFn}
F_n(z)=\sum\limits_{k}e_{k,n}f_kz^k,
\end{equation}
where \(e_{k,n}\) are the same as in \eqref{PoP} and \(f_k\) are the
Taylor coefficients of the function \(F(z)\) (See \eqref{TAY}).
Since roots of polynomial \(E_n\) lie in the left half-plane and the
\(F(z)\) belongs to the Laguerre-P\'{o}lya class
\(\mathscr{LP}\)-\textup{I},  it follows from P\'{o}lya-Schur Composition
Theorem\footnote{The Polya-Schur composition theorem appeared in the
paper \cite{PS}. It is based on the papers \cite{Sch1} by I.Schur and the
paper \cite{Po1} by G.Polya.}
that all roots of each of the polynomials \(F_n(z)\) lie
in the left half-plane \(\textup{Re}\,z<0\).
(See \cite{L1}, Chapt.8, Sect.3). Using the methods of \cite{L1}, Chapt.8, Sect.3, it follows that the sequence of the polynomials \(F_n(z)\) converges to the function \(F(z)\). Thus, all roots of \(F(z)\) lie in the
closed left half-plane, \(\textup{Re}\,z\leq0\).

It is possible to go one step further and prove the following
statement: If all roots of the function \(E_{\psi}(z)\) lie within the \emph{open}
left half-plane \(\textup{Re}\,z<0\), then so do all roots of the function \(F_{\psi}(z)\). We comment no further on these ideas here, but intend to address them in the very near future.

The first part of Theorem \ref{Te2} is clear. The second part of the theorem can be shown
in much the same way. Now considering the function \(F_{\psi}^{-}\), we
use the representation \eqref{InReFm}. The function \(F^{\prime}(z)\)
now serves the same purpose that \(F(z)\) did in our discussion of
\(F_{\psi}\). Since the function \(F\) belongs to the class \(\mathscr{LP}\)-\textup{I}, its derivative \(F^{\prime}(z)\) belongs
to the class \(\mathscr{LP}\)-\textup{I} as well.
\end{proof}

\noindent
\textbf{Acknowledgement}\\
I thank Professor Bernd Kirstein for his careful reading and editing
of the manuscript. His useful remarks and suggestions resulted in essential
improvements to our presentation of the material.

I thank Armin Rahn for his careful reading of the manuscript and his help
in improving the English in this paper.

%%%%%%%%%%%%%%%%%%%%%%%%%%%%%%%%%%%%%%%%%%%%%%%%%%%%%%%%%%%%%%%%%%%%%%%%%%%%%%%%%%%%%
%%%%%%%%%%%%%%%%%%%%%%%%%%%%%%%%%%%%%%%%%%%%%%%%%%%%%%%%%%%%%%%%%%%%%%%%%%%%%%%%%%%%%

\end{document}